\documentclass[11pt]{amsart}

 \usepackage[left=2.50cm, right=2.50cm, top=3.00cm, bottom=3.00cm]{geometry}
\usepackage{amsmath, amsthm, amssymb}
\usepackage{amsfonts}
\usepackage{hyperref}
\usepackage{cleveref}
\hypersetup{hidelinks}
\hypersetup{colorlinks, linkcolor=blue, urlcolor=blue, citecolor=blue}
\usepackage{cite}
\def\eqref#1{\textcolor{blue}{(\ref{#1})}}
\def\equationautorefname~#1\null{Eq.~(#1)\null}

\usepackage{aliascnt}

\def\eqref#1{\textcolor{blue}{(\ref{#1})}}

\usepackage{url} 
\usepackage{graphicx}
\usepackage{xcolor}
\usepackage{moreverb}

\usepackage{fancyvrb}

\def\r{\mathbb R}

\usepackage{listings}
 \lstnewenvironment{code}[1][]
  {\lstset{#1}}
  {}

\def\eqref#1{\textcolor{blue}{(\ref{#1})}}

\newtheorem{theorem}{Theorem}[section]
 \newtheorem{proposition}[theorem]{Proposition}
 
\theoremstyle{definition}

\author{Muhittin Evren Aydin$^1$, 
	Ayla Erdur Kara$^{2}$  }{}			

\address{
	$^1$Department of Mathematics, Faculty of Science, Firat University, 23200, Elazig, Turkey. \newline
	$^{2}$Department of Mathematics, Faculty of Science and Art, Tekirdag Namik Kemal University, 59030, Tekirdag, Turkey. 
}
\begin{document}

\title{Singular Minimal Ruled Surfaces}

\begin{abstract}

In this paper we study surfaces with minimal potential energy under gravitational forces, called singular minimal surfaces. We prove that a singular minimal ruled surface in a Euclidean $3-$space is cylindrical, in particular as an $\alpha-$catenary cylinder by a result of L\'{o}pez [Ann. Glob. Anal. Geom. 53(4) (2018), 521-541]. This result is also extended in Lorentz-Minkowski $3-$space.
\end{abstract}
\maketitle

\section{Introduction} \label{intro}

\quad Let $\mathbb{R}^3$ be a Euclidean 3–space with the canonical metric $\langle , \rangle$ and $(x,y,z)$ the rectangular coordinates. Denote by $X:M^2\to\mathbb{R}^3_+(\vec{v})$ a smooth immersion of an oriented surface $M^2$ into the halfspace 
\begin{equation*}
	\mathbb{R}^3_+(\vec{v})=\left\{q \in \mathbb{R}^3, \langle q, \vec{v}\rangle >0\right\},
\end{equation*}
for a given unit vector $ \vec{v} \in \mathbb{R}^3$. The potential $\alpha$--energy of $X$ in the direction $ \vec{v}$ is defined as 
 \begin{equation*}
 	E(X)=\int_{M^2} \langle X(q), \vec{v}\rangle^{\alpha}dM^2,\quad q\in M^2,
 \end{equation*} 
where $dM^2$ denote the measure on $M^2$ with respect to the induced metric tensor from $\mathbb{R}^3$. If $X$ is the critical point of the energy $E$, then the mean curvature $H$ of $M^2$ verifies
\begin{equation}\label{smequation}
	2H=\alpha\frac{\langle N, \vec{v}\rangle}{\langle X, \vec{v} \rangle},
\end{equation}
where $N$ denotes the unit normal vector field of $M^2$. A surface $M^2$ with \eqref{smequation} is called a {\it singular minimal surface} or $\alpha-${\it minimal surface} \cite{Dierkes2003singular}. The trivial case $\alpha=0$ leads $M^2$ to the well-known minimal surface (see \cite{Lopez2013constant}).

\quad Since the 1980s, the concept of singular minimal surfaces, one of the important topics in geometric analysis, has been of interest, originating in the study of the isoperimetric problem in $\mathbb{R}^3$ where a surface $M^2$ with minimal potential energy under gravitational forces is to be found (see \cite{bht}). This surface $M^2$ becomes the so-called {\it two-dimensional analogue of the catenary} desribed by Eq. \eqref{smequation} if the $z-$axis indicates the direction of gravity and if $\alpha =1$ and $\vec{v}=(0,0,1).$ Nitsche \cite{n} gave a necessary condition for the existence of the regular extremals. The corresponding variational problem was solved by Bemelmans and Dierkes \cite{bd}. In higher-dimensional case, Dierkes and Huisken \cite{dh} determined the conditions of (non-)existence of the problem. In addition, for energy minimizing hypersurfaces, a result of Bernstein type was showed in \cite{d1}. More recently, L\'{o}pez has published a number of papers on singular minimal surfaces \cite{Lopez2018invariant,Lopez2019,Lopez2019-1,Lopez2020-1,Lopez2020thetwo}, including the classification results, the existence and uniqueness of solutions to the Dirichlet problem, and the study of the shape of the compact ones using their boundaries. Because the physical relevance of the singular minimal surfaces, they find applications in various disciplines such as physics, engineering and architecture \cite{Giaquinta2004calculus, Gill2005catenary, Pottman2007architectural}. 

\quad Let $\gamma=\gamma(s)$ be a parameterized curve in $\mathbb{R}^2$. Recall that $\gamma$ is called $\alpha-${\it catenary} if satisfies (see \cite{Dierkes2003singular,Lopez2018invariant}) 
\begin{equation*}
	\kappa(s)=\alpha\frac{\langle n,\vec{v}\rangle}{\langle \gamma, \vec{v} \rangle}, \quad \alpha \in \mathbb{R},
\end{equation*}
where $\kappa$ and $n$ are the curvature and principal normal vector field of $\gamma$, respectively. The case that $\alpha=1$ refers to the catenary.

\quad Let $\mathbb{L}^3$ denote Lorentz-Minkowski space with the canonical Lorentzian metric $\langle ., .\rangle_L=dx^2+dy^2-dz^2$. L\'{o}pez \cite{Lopez2020thetwo} extended in $\mathbb{L}^3$ the notion of the singular minimal surfaces. In the Lorentzian setting, however, some problems arise immediately. For example, the $z-$coordinate now represents the time coordinate and so the concept of gravity has no more sense. Hence, the counterpart in $\mathbb{L}^3$ of  Eq. \eqref{smequation} is to characterize surfaces with prescribed angle between the normal vector field and a fixed direction. The other problem arises from three different causal characters of the surfaces, each of which behaves completely. Our interest in $\mathbb{L}^3$ is only for the spacelike surfaces, since we would like to keep the Riemannian setting. In addition, for a spacelike surface, the normal vector field is timelike. Since the $z-$axis is also timelike, the hyperbolic angle between two timelike vectors makes sense (see \cite{Lopez2014differential,o}).

\quad Let $\vec{v} \in \mathbb{L}^3$ be a fixed timelike vector and $X$ a smooth immersion of a spacelike surface $M^2$ into the half-space $\mathbb{L}^3_{+}(\vec{v})$. Then, $M^2$ is called a $\alpha-${\it singular maximal surface} if the mean curvature $H$ verifies
\begin{equation}\label{smequationlorentz}
	H=\alpha\frac{\langle N,\vec{v}\rangle_L}{\langle X,\vec{v}\rangle_L}, 
\end{equation}
where $N$ is the timelike normal vector field of $M^2$ \cite{Lopez2020thetwo}.

\quad On the other hand, a ruled surface is a one-parameter family of straight lines and locally accepts the  parametrization
\begin{equation*}
	X(s,t)=\gamma(s)+tw(s), \quad s\in I \subset\mathbb{R}, \quad  t\in \mathbb{R},
\end{equation*}
where $\gamma(s)$ is a regular curve and $w(s)$  a nonzero vector field with $\|w(s)\|=1$ for all $s\in I\subset\mathbb{R}$. The curve $\gamma(s)$ is said to be the {\it base} (or the {\it directrix}) of the surface and a line admitting $w(s)$ as a direction vector is called a {\it ruling} of the surface. Since the notion of ruled surface is affine and not metric, the above parametrization is also valid for $\mathbb{L}^3$. If all rulings are parallel to a fixed direction, i.e. $w^{\prime}(s)=0$ for every $s$, then the surface is cylindrical. 

\quad We found our motivation in study of the mean and the inverse mean curvature flows. Explicitly, Hieu and Hoang \cite{hh} proved that a ruled translating soliton to the mean curvature flow in $\mathbb{R}^3$ must be cylindrical. In contrast, the first author and L\'{o}pez \cite{al} showed the existence of non-cylindrical ruled translating solitons in the Lorentzian setting, which has no Euclidean counterpart. The same situation also arises for the translating solitons to the inverse mean curvature flow (see \cite{kp,ns}). Inspired by the cited works, we will investigate the singular minimal and maximal surfaces when they are ruled. We prove in the Euclidean setting that a singular minimal ruled surface is an $\alpha-$catenary cylinder (Theorem \ref{theorem-E}). This result is also extended in Lorentz-Minkowski space, but it is not a direct consequence of the Euclidean case. Because now there are two separate cases to investigate: $w'$ is non-degenerate (see Theorem \ref{L-1}) or lightlike (Theorem \ref{L-2}). At this point we would also like to emphasise that our consequences differ from the classical result that besides the planes, the minimal ruled surfaces are only the catenoids \cite{gray}.

\section{ Preliminaries}
\quad In this section we will express an auxiliary formula so that an immersed surface in $\mathbb{R}^3$ (resp. $\mathbb{L}^3$) can be singular minimal (resp. maximal).  We will mention only the basics of $\mathbb{L}^3$ because, having obtained the formula (see Eq. \eqref{sin.min.lor.}) explicitly in $\mathbb{L}^3$, we will understand that it is also valid in $\mathbb{R}^3$ with a small difference.

\quad A vector $\vec{v}\in \mathbb{L}^3 $ is called \emph{spacelike} if $\langle \vec{v},\vec{v}\rangle_L>0$ or $\vec{v}=0$, \emph{timelike} if $\langle \vec{v},\vec{v}\rangle_L<0$ and \emph{null} (\emph{lightlike}) if  $\langle \vec{v},\vec{v}\rangle_L=0$ and $\vec{v}\neq0$. If $\vec{v}$ is a timelike vector, then the \emph{timelike cone} including $v$ is defined by
$$
T(\vec{v})= \{\vec{w} \in \mathbb{L}^3:\langle \vec{v},\vec{w} \rangle_L<0\}.
$$
Let $\vec{v},\vec{w}$ be two vectors included in the same timelike cone. The {\it hyperbolic angle} between them is defined as a number $\theta$ such that
$$
\langle \vec{v},\vec{w} \rangle_L =-\| \vec{v}\|_L\| \vec{w}\|_L \cosh \theta,
$$
where $\|\vec{v}\|_L=\sqrt{|\langle \vec{v},\vec{v} \rangle_L|}$.

\quad The concepts of orthogonality and orthonormality is defined in the same way as in $\mathbb{R}^3$.  The characterization of subspaces in $\mathbb{L}^3$ depending on the causal characters is the following: A vector $\vec{v}\in \mathbb{L}^3$  is timelike (resp. spacelike) if and only if $\langle \vec{v} \rangle^{\perp}$ is spacelike (resp. timelike) and so $\mathbb{L}^3=\langle \vec{v} \rangle \oplus \langle \vec{v} \rangle ^{\perp} $. A subspace $U\subset \mathbb{L}^3$ is spacelike (resp. timelike or lightlike) if and only if $U^{\perp}$ is timelike (resp. spacelike or lightlike).

 \quad The Lorentzian cross product of two vectors $\vec{u},\vec{v} \in \mathbb{L}^3$  is defined as the unique vector $\vec{u}\times_L \vec{v}$ such that $\langle \vec{u}\times_L \vec{v},\vec{w}\rangle_L=\left(\vec{u},\vec{v},\vec{w}\right)$ for every $\vec{w} \in \mathbb{L}^3$, where $\left(\vec{u},\vec{v},\vec{w}\right)$ denotes the determinant of the $3\times 3$ matrix formed by the vectors $\vec{u},\vec{v}$ and $\vec{w}$.
 
 \quad Let $ M^2$ be a surface isometrically immersed in $\mathbb{L}^3$ whose induced metric is non-degenerate. If the metric is \textit{Riemannian} (resp. \textit{Lorentzian}), the surface is said to be\textit{ spacelike} (resp. \textit{timelike}). For a parametrization $X=X(s,t)$ of $ M^2$, the unit normal vector field  $N$ of $X$ is defined by 
\begin{equation*}
	N=\frac{X_s \times_L X_t}{\|X_s \times_L X_t\|_L}.
\end{equation*}

\quad The coefficients of the first fundamental form and the second fundamental form are, respectively,
 \begin{equation*}
 	E=\langle X_s,X_s \rangle_L, \quad F=\langle X_s,X_t \rangle_L, \quad  G=\langle X_t,X_t \rangle_L
 \end{equation*}
and 
\begin{eqnarray*}
	e=\langle N,X_{ss} \rangle_L&=&\frac{\left(X_s,X_t,X_{ss}\right)}{\sqrt{\left|EG-F^2\right|}},\\
	f=\langle N,X_{st} \rangle_L&=&\frac{\left(X_s,X_t,X_{st}\right)}{\sqrt{\left|EG-F^2\right|}},\\
	g=\langle N,X_{tt} \rangle_L&=&\frac{\left(X_s,X_t,X_{tt}\right)}{\sqrt{\left|EG-F^2\right|}},
\end{eqnarray*}
where $\sqrt{\left|EG-F^2\right|}=\|X_s \times_L X_t\|_L$.

\quad Set $\langle N,N\rangle_L=\epsilon$. If $N$ is spacelike (resp. timelike), then $\epsilon=1$ (resp. $\epsilon=-1$). Hence $\epsilon=-1$ (resp. $\epsilon=1$) if  $ M^2 $ is spacelike (resp. timelike). Noting that $|EG-F^2|=-\epsilon(EG-F^2)$, the mean curvature is
\begin{equation*}\label{Lorentzmeancurvature}
	H=-\frac{1}{2}\frac{G(X_s,X_t,X_{ss})-2F(X_s,X_t,X_{st})+E(X_s,X_t,X_{tt})}{|EG-F^2|^{3/2}},
\end{equation*}
Hence,  Eq. \eqref{smequationlorentz} is now
\begin{equation}\label{sin.min.lor.}
	G(X_s,X_t,X_{ss})-2F(X_s,X_t,X_{st})+E(X_s,X_t,X_{tt})=\epsilon \alpha\frac{EG-F^2}{\langle X,\vec{v} \rangle}(X_s,X_t,\vec{v}) .
\end{equation}
For the Euclidean setting,  Eq. \eqref{sin.min.lor.} is still valid discarding $\epsilon$.

\section{Singular Minimal Ruled Surfaces}

\quad In this section we will classify singular minimal surfaces when they are ruled. Recall that an $\alpha-${\it catenary cylinder} is a cylinder which admits the base curve as a $\alpha-$catenary. By the result of L\'{o}pez \cite[Theorem 1]{Lopez2018invariant}, if a singular minimal surface is cylindrical then it is either a plane parallel to $\vec{v}$ or a $\alpha-$catenary cylinder whose rulings are orthogonal to $\vec{v}$. This result plays a key role in our study.

\quad  We will be interested in the ruled surfaces $M^2$ in $\mathbb{R}^3$ which are non-cylindrical. Hence, the surface $M^2$ parameterizes as
\begin{equation} \label{teo1.1}
X(s,t)=\gamma(s)+tw(s), \quad s\in I\subset\mathbb{R}, \quad t\in \mathbb{R},
\end{equation}
where $w'(s) \neq 0$ for every $s \in I$. By reparameterization we may assume that 
\begin{equation} \label{ruled}
\langle \gamma^{\prime}(s), w(s)\rangle=\langle\gamma^{\prime}(s),w^{\prime}(s) \rangle=0, \quad  \langle w(s),w(s) \rangle=\langle w^{\prime}(s),w^{\prime}(s)\rangle=1 ,
\end{equation} 
for every $s \in I$  (see \cite[Proposition 3.4]{Ch01oi20classification}).
  
\quad We now introduce the functions $P(s)$ and $Q(s)$ defined on $ I $:
\begin{equation*}\label{teo1.2}
	P(s)=( w(s), w'(s), \gamma'(s) ) , \quad Q(s)=( w(s), w'(s), w''(s)).
\end{equation*}
Since $\langle w(s), w'(s)\rangle =0$ for every $s$, we have the orthonormal basis $\left\{w,w^{\prime},w\times w^{\prime}\right\}$. In terms of the basis, 
\begin{eqnarray}
	\gamma^{\prime}(s)&=&P(s)w(s)\times w^{\prime}(s),\label{teo1.3}\\
	\gamma^{\prime}(s)\times w(s)&=&P (s)w^{\prime}(s), \label{teo1.4}\\
	w^{\prime\prime}(s)&=&-w(s)+Q(s)w(s)\times w^{\prime}(s), \label{teo1.5}\\
	w(s)\times w^{\prime\prime}(s)&=&-Q(s)w^{\prime}(s),\label{teo1.6}
\end{eqnarray}
where $ P(s) \neq 0$ on $I$. Because $\langle w(s), w(s) \rangle =1$, the vector field $w(s)$ defining the rulings can be viewed as a curve lying on a unit sphere $\mathbb{S}^2$ in $\r^3$ and so the function $Q(s)$ plays the role of its geodesic curvature.

\quad Some quantities are
\begin{eqnarray*}
	E&=&P^2+t^2,\quad F=0,\quad G=1,\\
	X_s\times X_t&=&Pw^{\prime}-tw\times w^{\prime},\\
	X_{ss}&=&-tw-PQw^{\prime}+(P^{\prime}+tQ)w\times w^{\prime},\\
	(X_s,X_t,X_{ss})&=&-P^2Q-P^{\prime}t-Qt^2.
\end{eqnarray*}

\quad In what follows, we obtain the non-existence of singular minimal ruled surfaces, except the cylindrical ones.
\begin{theorem} \label{theorem-E}
The planes and $\alpha$--catenary cylinders in $\mathbb{R}^3$  are the only singular minimal ruled surfaces. 
\end{theorem}
\begin{proof}
The proof is by contradiction. Let $M^2$ be a non-cylindrical ruled surface parametrized by Eq. \eqref{teo1.1} where the equalities in Eq. \eqref{ruled} hold. Assume that $M^2$ is singular minimal. Then Eq. \eqref{sin.min.lor.} is now 
\begin{equation}\label{teo1.8}
	\sum_{n=0}^{3}A_nt^n=0,
\end{equation}
where 
\begin{eqnarray}
	A_0&=&\alpha P^3\langle w^{\prime},\vec{v}\rangle+P^2Q\langle \gamma,\vec{v} \rangle,\label{teo1.9}\\
	A_1&=&-\alpha P^2( w , w^{\prime},\vec{v})+P^2Q\langle w,\vec{v} \rangle+P^{\prime}\langle \gamma,\vec{v} \rangle,\label{teo1.10}\\
	A_2&=&\alpha P\langle w^{\prime},\vec{v}\rangle+Q\langle \gamma,\vec{v} \rangle+P^{\prime}\langle w,\vec{v} \rangle,\label{teo1.11}\\
	A_3&=&-\alpha ( w , w^{\prime},\vec{v})+Q\langle w,\vec{v}\rangle.\label{teo1.12}
\end{eqnarray}
Since Eq. \eqref{teo1.8} is a polynomial equation in $ t $, the coefficient functions $ A_0,...,A_3 $ are identically $0$. From Eqs. \eqref{teo1.9} and \eqref{teo1.11}, we conclude $P^{\prime}\langle w,\vec{v}\rangle=0$. Analogously, from Eqs. \eqref{teo1.10} and \eqref{teo1.12} we have $	P^{\prime}\langle \gamma,\vec{v}\rangle=0$. Hence, one yields $ P^{\prime}=0 $ because otherwise one contradicts with $\langle X, \vec{v}\rangle >0$.  Set $ P(s)=P_0\neq0 $. We claim that
\begin{equation}
	Q(s)\neq 0, \quad \left\langle \gamma ,\vec{v}\right\rangle \neq 0, \quad \left\langle
	w ,\vec{v}\right\rangle \neq 0, \quad \text{for every $s \in I$.} \label{teo1.13}
\end{equation}
The proof of our claim is as follows:
\begin{enumerate}
\item Case $Q=0$. From Eqs.\eqref{teo1.9} and \eqref{teo1.12} we have $\left\langle w ^{\prime },\vec{v}\right\rangle =0$ and $	( w ,w ^{\prime },\vec{v})=0$, respectively. This gives $w \left( s\right) $ parallel to $\vec{v}$ for every $s\in I$. Being $\vec{v}$ fixed, the cylindrical case, i.e. $w \left( s\right) =\pm \vec{v}$, appear, which was initally discarded. 
\item Case $\left\langle\gamma ,\vec{v}\right\rangle =0$. Differentiating, $\left\langle \gamma ^{\prime },%
\vec{v}\right\rangle =0$ or $P_{0}( w , w ^{\prime },%
\vec{v})=0$ due to Eq. \eqref{teo1.3}. It follows from Eq. \eqref{teo1.12}
that $\left\langle w ,\vec{v}\right\rangle =0,$ contradicting $\langle X, \vec{v}\rangle >0$.
\item Case $\left\langle w ,\vec{v}\right\rangle =0$. Since $\left\langle w
^{\prime },\vec{v}\right\rangle =0$, we deduce $\vec{v}$ parallel to $w
\times w ^{\prime },$ that is, $w \times w ^{\prime }=\pm \vec{v}.$
This gives from Eq. \eqref{teo1.12} that $\mp \alpha \left\langle \vec{v},\vec{v}\right\rangle =0$, which is not possible.
\end{enumerate}
\quad We now proceed the proof, writing Eqs.  \eqref{teo1.9} and \eqref{teo1.12} respectively as
\begin{eqnarray}
	-\alpha P_{0}\left\langle w ^{\prime },\vec{v}\right\rangle 
	&=&Q\left\langle \gamma ,\vec{v}\right\rangle ,  \label{teo1.14} \\
	\alpha ( w ,w ^{\prime },\vec{v})
	&=&Q\left\langle w ,\vec{v}\right\rangle .  \label{teo1.15}
\end{eqnarray}%
Since all the terms in Eqs. \eqref{teo1.14} and \eqref{teo1.15} are nonzero, we have%
\begin{equation}
	P_{0}\left\langle w ^{\prime },\vec{v}\right\rangle \left\langle w ,%
	\vec{v}\right\rangle +( w ,w ^{\prime },\vec{v}) \left\langle \gamma ,\vec{v}\right\rangle =0.  \label{teo1.16}
\end{equation}%
We differentiate Eq. \eqref{teo1.16} and next use Eqs. \eqref{teo1.3}-\eqref{teo1.6}, obtaining
\begin{equation}
	-P_{0}\left\langle w ,\vec{v}\right\rangle ^{2}+P_{0}\left(
	Q\left\langle w ,\vec{v}\right\rangle ( w ,w ^{\prime },\vec{v}) +( w ,w ^{\prime },\vec{v})^{2}\right) +P_{0}\left\langle w^{\prime },\vec{v}%
	\right\rangle ^{2}-Q\left\langle \gamma ,\vec{v}\right\rangle \left\langle
	w^{\prime },\vec{v}\right\rangle =0.  \label{teo1.17}
\end{equation}%
Substituting Eqs. \eqref{teo1.14} and \eqref{teo1.15} into Eq. \eqref{teo1.17}, we have
\begin{equation}
- \left\langle w,\vec{v}\right\rangle ^{2} + 	\left( 1+\alpha \right) \left ( \left\langle w ^{\prime },\vec{v}\right\rangle
	^{2}+( w ,w ^{\prime },\vec{v})^{2} \right) =0.  \label{teo1.18}
\end{equation}%
On the other hand, since $\vec{v}$ is unitary, 
\begin{equation}
 \left\langle w,\vec{v}\right\rangle ^{2} +  \left\langle w ^{\prime },\vec{v}\right\rangle
	^{2}+( w ,w ^{\prime },\vec{v})^{2}=1.  \label{teo1.19}
\end{equation}
From Eqs. \eqref{teo1.18} and \eqref{teo1.19} we have $(2+\alpha)\left\langle w,\vec{v}\right\rangle ^{2} =1+\alpha$. Differentiating, $\left\langle w',\vec{v}\right\rangle =0$ or $\left\langle \gamma,\vec{v}\right\rangle =0$ due to Eq. \eqref{teo1.14}. This contradicts with Eq. \eqref{teo1.13}.
\end{proof}

\section{Singular Maximal Ruled Surfaces}

\quad Since we are interested in a spacelike ruled surface $M^2$, a parameterization on it is
\begin{equation}\label{pm}
	X(s,t)=\gamma(s)+tw(s), \quad s\in I\subset \mathbb{R}, \quad t\in \mathbb{R}. 
\end{equation}
Here, because of the causal character of $M^2$, the only possibility is that both the base curve and the vector field which defines the rulings are spacelike \cite{cky}. Hence, $\left\langle \gamma' (s) , \gamma' (s) \right\rangle_L >0$ and $ \left\langle w(s), w(s) \right\rangle_L >0 $ for every $s  \in I$. 

\quad To simplify our calculations, we examine the reparameterization of a given spacelike non-cylindrical ruled surface, similar to the Euclidean case (see \cite[Proposition 3.4]{Ch01oi20classification}).

\begin{proposition} \label{propos}
Let $ M^2 $ be a spacelike non--cylindrical ruled surface in $\mathbb{L}^3$ with parametrization 
	\begin{equation*}
		X_1(s,t)=\gamma_1(s)+tw(s),
	\end{equation*}
for $\langle \gamma_1^{\prime},w\rangle_L=0$, $\langle w,w\rangle_L=1$ and $\langle w^{\prime},w^{\prime}\rangle_L=\pm 1$.  Then, there exists a reparametrization of $ M^2 $ such that
	\begin{equation*}
		X(s,t)=\gamma(s)+tw(s),
	\end{equation*}
where $\langle \gamma^{\prime},w\rangle_L=0$,  $\langle \gamma^{\prime},w^{\prime}\rangle_L=0$, $\langle w,w\rangle_L=1$ and $\langle w^{\prime},w^{\prime}\rangle_L=\pm 1$.
\end{proposition}
\begin{proof}
Let $\gamma_1$ be a spacelike curve and $ w $ a smooth vector field with $\langle w,w \rangle_L =1$. We may parameterize $ M^2 $ as 
$$
X_1(s,t)=\gamma_1(s)+tw(s),
$$ 
where $\langle \gamma_1^{\prime},w\rangle_L=0$, $\langle w,w\rangle_L=1$ and $\langle w^{\prime},w^{\prime}\rangle_L=\pm 1$. We set $f_1(s)=\langle \gamma_1(s),w(s) \rangle_L$, $f_2(s)=\langle \gamma_1^{\prime}(s),w^{\prime}(s) \rangle_L$ and $f_3(s)=\langle \gamma_1(s),w^{\prime}(s) \rangle_L$. Now, introduce a curve $\gamma$ by 
$$
\gamma(s)=y_1(s)\gamma_1(s)+y_2(s)w(s),
$$
where $y_1$ and $y_2$ are the solutions to the system of ordinary differential equations
	\begin{eqnarray*}
		f_1(s)y_1^{\prime}(s)+y_2^{\prime}(s)=0,\\
		f_2(s)y_1(s)+f_3(s)y_1^{\prime}(s) \pm y_2=0
	\end{eqnarray*}
with a proper initial condition $y_1(0)$ and $y_2(0)$. Then, we conclude $\langle \gamma^{\prime},w \rangle_L=0$ and $\langle \gamma^{\prime},w^{\prime} \rangle_L=0$, proving the result.
\end{proof} 

\quad Now we separate our investigation into two cases in which $ w^{\prime}(s) $ is lightlike or not. 

\begin{theorem}\label{L-1}
Let $ M^2$ be a spacelike non-planar ruled surface and $w(s)$ be a spacelike vector field in $\mathbb{L}^3$ that defines the rulings of $M^2$ such that $w'$ is non-degenerate. If $ M^2 $ is an $\alpha$-singular maximal surface, then $ M^2 $ is cylindrical.
\end{theorem}
\begin{proof}
By contradiction, we assume that $ M^2 $ is a non-cylindrical ruled surface being $\alpha$-singular maximal. By Proposition \ref{propos}, we may suppose on $I$ that
\begin{equation*}
\langle \gamma'(s),w(s) \rangle_L =0,\quad \langle \gamma'(s),w'(s) \rangle_L =0, \quad \langle w(s),w(s) \rangle_L=1, \quad \langle w'(s),w'(s)\rangle_L=\delta=\pm 1.
\end{equation*}
We have the orthonormal basis $\left\{w,w^{\prime},w \times_L w^{\prime}\right\}$, where $\langle w \times_L w^{\prime},w \times_L w^{\prime} \rangle=-\delta $. As in the Euclidean setting, we may introduce the two smooth functions
\begin{equation*}\label{teo3.1}
	P(s)=(\gamma'(s), w(s), w'(s)), \quad Q(s)=(w(s),w'(s),w''(s)).
\end{equation*}
In terms of the basis above, we get
\begin{eqnarray*}
	\gamma'&=& -\delta Pw\times_L w',\label{teo3.2}\\
	\gamma^{\prime}\times_L w&=&\delta Pw^{\prime}, \label{teo3.3}\\
	w''&=& - \delta(w + Qw\times_L w'),\label{teo3.4}\\
	w\times_L w''&=& - \delta Qw'.\label{teo3.5}
\end{eqnarray*}
A direct calculation follows
\begin{eqnarray*}
	E&=&-\delta (P^2 - t^2),\quad F=0,\quad G=1,\\
	X_s\times_L X_t&=&\delta Pw^{\prime}-tw\times_L w^{\prime},\\
	X_{ss}&=&-\delta tw+PQw^{\prime}-\delta( P^{\prime}+tQ)w\times_L w^{\prime},\\
	(X_s,X_t,X_{ss})&=&P^2Q-P^{\prime}t - Qt^2.
\end{eqnarray*}
Eq. \eqref{sin.min.lor.} is now
$$
\sum_{n=0}^{3}A_nt^n=0,
$$
where 
\begin{eqnarray*}
	A_0&=&-\alpha P^3\langle w^{\prime},\vec{v}\rangle_L+P^2Q\langle \gamma,\vec{v} \rangle_L,\\
	A_1&=& \alpha \delta P^2( w , w^{\prime},\vec{v})+P^2Q\langle w,\vec{v} \rangle_L - P^{\prime}\langle \gamma,\vec{v} \rangle_L,\\
	A_2&=&\alpha P\langle w^{\prime},\vec{v}\rangle_L -Q\langle \gamma,\vec{v} \rangle_L-P^{\prime}\langle w,\vec{v} \rangle_L,\\
	A_3&=& -\alpha \delta ( w , w^{\prime},\vec{v})-Q\langle w,\vec{v}\rangle .
\end{eqnarray*}
The result concludes with the same arguments from the proof of Theorem \ref{theorem-E}.
\end{proof}

\quad We now treat the case that $w$ is spacelike and $w'$ is lightlike on $I$. Then $ M^2 $ is parametrized as Eq. \eqref{pm} where now
$$
\langle \gamma(s),\gamma(s) \rangle_L=1 \quad  \langle\gamma'(s),w(s) \rangle_L=0, \quad \langle w(s),w(s) \rangle_L=1,
$$ 
for every $s \in I$ \cite{cky}. We have the orthonormal frame $\{ \gamma^{\prime},w,\gamma^{\prime} \times_Lw \}.$

\quad We notice that $w'(s)$ is a lightlike direction in the hyperboloid $\{p\in\mathbb{L}^3:\langle p,p\rangle_L=1\}$ and so $w(s)$ is a straight line \cite{Dillen1999ruled}. With a change of parameter, we may suppose $w(s)=\vec{a}s+\vec{b}$, where 
$$
\langle \vec{a},\vec{a}\rangle_L
=\langle \vec{a},\vec{b}\rangle_L =0, \quad \langle \vec{b},\vec{b}\rangle_L =1.
$$  
After  a rigid motion of $\mathbb{L}^3$, we can consider 
\begin{equation*}
	w(s)=(1,s,s)=s(0,1,1)+(1,0,0), \quad  \vec{a}=(0,1,1),\,  \vec{b}=(1,0,0).
\end{equation*}
With this choice above, we have $w\times_L w'=-w'$.

\quad Set $Q(s)=\langle \gamma^{\prime}(s), w^{\prime}(s) \rangle_L $. Here $Q$ is always different from $0$ because otherwise $w'$ would parallel to $\gamma^{\prime} \times_Lw$. This is impossible due to that $\gamma^{\prime} \times_Lw$ is timelike. Then, we have
\begin{eqnarray*}
	w^{\prime}&=&Q(\gamma^{\prime}+\gamma^{\prime}\times_L w),\label{teo4.1}\\
	\gamma^{\prime\prime}&=&-Qw-(\gamma',w,\gamma'')\gamma^{\prime}\times_Lw \label{teo4.2}.
\end{eqnarray*}
We notice here that
\begin{eqnarray*}
Q'&=&\langle \gamma'',w' \rangle_L \\
&=&(\gamma',w,\gamma'') \langle \gamma' , w \times_L w' \rangle_L \\
&=&Q  (\gamma',w,\gamma'') ,
\end{eqnarray*}
or equivalently, $ (\gamma',w,\gamma'')=Q'/Q$. In addition,
$$
E=1+2Qt, \quad F=0, \quad G=1.
$$
and 
\begin{eqnarray*}
	X_{ss}&=&-Qw- \frac{Q'}{Q} \gamma^{\prime}\times_Lw,\\
	X_s\times_LX_t&=&Qt\gamma^{\prime}+ (1+Qt )\gamma^{\prime}\times_Lw ,\\
	\left(X_s,X_t,X_{ss}\right)&=&\frac{Q^{\prime}}{Q}+Q^{\prime}t.
\end{eqnarray*}

\begin{theorem}\label{L-2}
Let $ M^2$ be a spacelike ruled surface and $w(s)$ be a spacelike vector field in $\mathbb{L}^3$ which defines the rulings of $M^2$ such that $w'$ is lightlike. If $ M^2 $ is an $\alpha-$singular maximal surface, then $ M^2 $ is cylindrical.
\end{theorem}
\begin{proof}
On the contrary, we assume that $ M^2 $  is an $\alpha-$singular maximal ruled surface which is non-cylindrical. 
So, Eq. \eqref{sin.min.lor.} is
\begin{equation*}
	\sum_{k=0}^{2}A_kt^k=0,
\end{equation*}
where
\begin{eqnarray}
	A_0&=&\frac{Q'}{Q}\langle\gamma, \vec{v} \rangle_L+\alpha  (\gamma',w,\vec{v}) ,\label{teo4.3}\\
	A_1&=&\frac{Q'}{Q}\langle w, \vec{v} \rangle_L +Q'\langle\gamma, \vec{v} \rangle_L + \alpha Q (\langle\gamma', \vec{v} \rangle_L +3(\gamma',w,\vec{v})),\label{teo4.4}\\
	A_2&=&Q'\langle w, \vec{v} \rangle_L +2\alpha Q^2(\langle\gamma', \vec{v} \rangle_L +(\gamma',w,\vec{v})).\label{teo4.5}
\end{eqnarray}
Here $A_0,A_1,A_2$ are identically $0$. Dividing Eq. \eqref{teo4.4} by Q and considering Eq. \eqref{teo4.3} into the resulting equation, we conclude
\begin{equation}
\frac{Q'}{Q^2}\langle w, \vec{v} \rangle_L + \alpha \langle\gamma', \vec{v} \rangle_L +2\alpha (\gamma',w,\vec{v})=0. \label{last}
\end{equation}
Now we divide Eq. \eqref{teo4.5}  by $Q^2$ and replace in the last equation, obtaining $\langle\gamma', \vec{v} \rangle_L=0.$ So, we have $\vec{v}=a(s)w(s)+b(s)\gamma'\times_Lw$, where $b(s)$ is always different from $0$ because $\vec{v}$ is timelike but $w(s)$ is spacelike. This immediately implies $Q'\neq 0$. Next, Eqs. \eqref{teo4.3} and \eqref{last} are now
\begin{equation}
Q' \langle\gamma, \vec{v} \rangle_L+\alpha Qb=0, \label{last2}
\end{equation}
and 
\begin{equation}
Q'a +2\alpha Q^2b=0. \label{last3}
\end{equation}
By Eqs. \eqref{last2} and \eqref{last3}, we have $\langle\gamma, \vec{v} \rangle_L=a/Q$. Replacing in Eq. \eqref{teo4.3},
$$
Q'a+\alpha Q^2b=0.
$$
Comparing with Eq. \eqref{last3}, we arrive to the contradiction $\alpha Q^2b=0.$
\end{proof}


 \end{document}